\documentclass[11pt]{amsart}

\usepackage{epigamath}

\usepackage[english]{babel}

\numberwithin{equation}{section}

\usepackage{amsthm, amsfonts, amssymb, amsmath, latexsym, enumerate,array}
\usepackage{amscd} 
\usepackage[all]{xy}
\usepackage{xcolor}
\usepackage{pgf,tikz}
\usetikzlibrary{arrows}
\graphicspath{{.}}
   \makeatletter
   \let\accent@spacefactor\relax
   \makeatother
\usepackage{amsfonts}
\usepackage{amsmath}   
\usepackage{amsthm}    
\usepackage{amscd}     
\usepackage{euscript}

\usepackage[shortlabels]{enumitem}
\setlist[enumerate,1]{label={\rm(\arabic*)}, ref={\rm\arabic*}}

\usepackage{calc}

\newtheorem{theo}{Theorem}[section]
\newtheorem*{theo*}{Theorem}
\newtheorem{pro}[theo]{Proposition}

\newtheorem{cor}[theo]{Corollary}

\theoremstyle{definition}
\newtheorem{defi}[theo]{Definition}

\theoremstyle{remark}
\newtheorem*{question}{Question}
\newtheorem{exa}[theo]{Example}

\newtheorem{remark}[theo]{Remark}

\def\cA{{\mathcal A}}
\def\TA{{\mathcal T}_{\mathcal A}}

\def\T{{\mathcal T}}

\def\C{{\mathbb C}}

\def\cO{{\mathcal O}}
\def\HH{{\rm H}}
\def\hh{{\rm h}}
\def\P{{\mathbb P}}
\def\PD{\check{\mathbb P}}
\def\F{{\mathbb F}}

\DeclareMathOperator{\sing}{Sing}
\DeclareMathOperator{\Supp}{Supp}
\DeclareMathOperator{\Tr}{Tr}
\DeclareMathOperator{\rem}{rem}

\newcommand{\Mod}[1]{\ (\mathrm{mod}\ #1)}

\newcommand{\supth}[1]{\ensuremath{#1^{\mathrm{th}}}}

\EpigaVolumeYear{7}{2023} \EpigaArticleNr{14} \ReceivedOn{April 1, 2021}
\InFinalFormOn{October 2, 2022}
\AcceptedOn{January 15, 2023}

\title{Triangular arrangements on the projective plane}
\titlemark{Triangular arrangements on the projective plane}

\author{Simone Marchesi}
\address{Universitat de Barcelona, 
Departament de Matem\`{a}tiques i Inform\`{a}tica, 
 Gran Via de les Corts Catalanes, 585, 
 08007 Barcelona, Spain}
\email{marchesi@ub.edu}
\author{Jean Vall\`es}
\address{Universit\'e de Pau et des Pays de l'Adour, 
  LMAP-UMR CNRS 5142, 
  Avenue de l'Universit\'e - BP 1155.
  64013 Pau Cedex, France}
\email{jean.valles@univ-pau.fr}

\authormark{S.~Marchesi and J.~Vall\`es}

\AbstractInEnglish{In this work we study line arrangements consisting in lines passing through three non-aligned points.  We call them triangular arrangements. We prove that any combinatorics of a triangular arrangement is always realized by a Roots-of-Unity-Arrangement, which is a particular class of triangular arrangements. Among these Roots-of Unity-Arrangements, we provide conditions that ensure their freeness.  Finally, we give two triangular arrangements having the same weak combinatorics, such that one is free but the other one is not.}


\acknowledgement{We would like to thank the University of Campinas and the Universit\'e de Pau et des Pays de l'Adour for the hospitality and for providing the best working conditions.
The first author was partially supported by Funda\c{c}\~{a}o de Amparo \`{a} Pesquisa do Estado de S\~{a}o Paulo (FAPESP) grant 2017/03487-9, by  CNPq grant 303075/2017-1 and by PID2020-113674GB-I00 (Ministerio de Ciencia e Innovación - Spain.)
The second author was partially supported by FAPESP grant 2018/08524-2.}

\begin{document}


\maketitle

\begin{prelims}

\DisplayAbstractInEnglish

\bigskip

\DisplayKeyWords

\medskip

\DisplayMSCclass

\end{prelims}


\newpage

\setcounter{tocdepth}{1}

\tableofcontents


\section{Introduction}
A line arrangement $\cA=\{l_1,\ldots, l_n\}$ in $\P^2$ is a finite set of distinct lines.  The union of these lines forms a divisor 
defined by an equation $f=\prod_i f_i=0$, where $f_i=0$ is the equation defining $l_i$. The sheaf $\T_{\cA}$ of vector fields 
tangent to this arrangement can be
defined as the kernel of the Jacobian map, which means
\[\xymatrix@C=16pt{
  0\ar[r]&\T_{\mathcal{A}}\ar[r]&\cO_{\P^2}^3 \ar[r]^-{\nabla f} & \mathcal{J}_f(n-1)\ar[r]&0,}
\]
where $\mathcal{J}_f$ is the ideal sheaf generated by the three partial derivatives $\nabla f=(\partial_x f,\partial_y f, \partial_z f)$. This ideal, called the\, \textit{Jacobian ideal}, defines the Jacobian scheme supported by the singular points of the arrangement; for instance, when $\cA$ is generic (\textit{i.e.}~it consists of $n$ lines in general position), $\mathcal{J}_f$ defines $\binom{n}{2}$ distinct points.  The sheaf $\T_{\cA}$, which is a reflexive sheaf over $\P^2$ and therefore a vector bundle, is a basic tool to study the link between the geometry, the topology and the combinatorics of $\cA$.

When $\T_{\cA}=\cO_{\P^2}(-\alpha)\oplus \cO_{\P^2}(-\beta)$, the arrangement $\cA$ is called a \textit{free arrangement}, and the integers $(\alpha,\beta)$ are called its exponents. This notion of freeness was first introduced by Saito \cite{S} for a reduced divisor and by Terao \cite{T} for hyperplane arrangements.

In \cite{OT} Terao conjectures that freeness depends only on the combinatorics of $\cA$, where the combinatorics is described by the intersection lattice $L(\cA)$, that is, the set of all the intersections of lines in $\cA$ (see Section~\ref{sec-RUA}).  More precisely, if two arrangements $\cA_0$ and $\cA_1$ have the same combinatorics (a bijection between $L(\cA_0)$ and $L(\cA_1)$) and one of them is free, then the other one is also free (of course with the same exponents).  This conjecture, despite all the efforts, has been proved, for line arrangements, only up to $13$ lines (see \cite{DIM}).

A weaker problem concerns the weak combinatorics (see Section~\ref{sec-weak}).
The weak combinatorics of a given arrangement of $n$ lines is given by the knowledge of the integers $t_i$, $i\ge 2$, which are the numbers of points with multiplicity exactly  $i$ of the arrangement. Let us mention 
the following formula, proved by Hirzebruch in \cite{Hirzebruch}, involving these numbers (when $t_n=t_{n-1}=t_{n-2}=0$):
 \begin{equation}
  \mbox{$t_2+t_3\ge n+\sum_{i\ge 1}it_{i+4}. $}    
  \end{equation}

In this paper we consider \textit{triangular arrangements}, meaning that there exist three non-aligned points such that every line of the arrangement passes through one of 
these points; for instance, the reflection arrangements described in \cite{OT} and again in \cite{Hirzebruch} belong to this family.

The paper is organized in the following way.

In Section~\ref{sec-innerpt} we will prove preliminary results on triangular arrangements and introduce the concept of a \textit{Roots-of-Unity-Arrangement} (RUA), which form, in our opinion, a central family to be studied in the context of arrangements, as confirmed by the following results.

Indeed, in Section~\ref{sec-RUA} we prove a key result, which states that for any triangular arrangement, it is possible to find a RUA  with exactly the same combinatorics.

  Combinatorial conditions proving freeness for a RUA are given in Section~\ref{sec-family}.  The main idea is to study together with a RUA $\cA$, obtained by removing inner lines from a free RUA (more precisely, from the so-called full monomial arrangement $\cA_3^3(N)$ introduced in \cite[Proposition 6.77]{OT}), its \textit{complementary arrangement} $\cA^{c}$ consisting of the removed lines. This is still a triangular arrangement, and its set of inner triple points (intersections of three lines, one from each vertex) will be denoted by $T_{\rem}$.  When this set is empty or when it reaches particular values (see Theorem~\ref{Trem-+}), the arrangement $\cA$ will be free. As a direct consequence, we describe explicitly a free \textit{equilateral} RUA (where by equilateral we mean that it has the same number of lines passing through each vertex) for any possible pair of exponents (see Corollary ~\ref{equilateral}).

 In Section~\ref{sec-novertex} we study freeness for triangular arrangement when one or two sides of the triangle (the one defined by three non-aligned points through which all lines pass) are removed. Contrarily to the previous section, few pairs of exponents are possible. They are obtained when the set of inner triple points (intersections of three lines of the arrangement that are not sides) is a complete intersection.

In Section~\ref{sec-weak}, we show that Terao's conjecture does not extend to the assumption of weak combinatorics. Indeed, we prove the following.

\begin{theo*}
 There exist pairs of arrangements possessing the same weak combinatorics such that one is free and the other is not.
\end{theo*}

Finally, in Section~\ref{sec-Terao} we discuss Terao's conjecture for triangular arrangements in general.

\section{The inner triple points of a triangular arrangement}\label{sec-innerpt}
In this section we will explain the importance of the set $T$ of the triple points, defined by the triangular arrangement, which are not the vertices of the triangle. 
We will describe in particular the case where $T$ is either empty or a complete intersection. 

Let $A,B,C$ be three non-aligned points. A line arrangement such that any of its lines passes through $A,B$ or $C$ is called \textit{triangular arrangement}.  The lines $(AB)$, $(AC)$ and $(BC)$, sides of the triangle $ABC$, are called \textit{side lines}. If one of the three side lines is missing, we will say that the triangular arrangement is incomplete. The set of triangular arrangements consisting of $a+1$ lines through $A$, $b+1$ lines through $B$, $c+1$ lines through $C$, where $0\le a\le b\le c$, including the three side lines, will be denoted by $\Tr(a,b,c)$; these arrangements possess $a+b+c$ lines.

\begin{defi}
 Let $A$, $B$ and $C$ be three non-aligned points in $\P^2$. Let us consider a triangular arrangement $\cA$ with respect to these three points.  A singular point of $\cA$ which appears as an intersection of one line through $A$, one line through $B$ and one line through $C$ is called an \textit{inner triple point} of $\cA$. The set of these inner triple points of $\cA$ is denoted by $T$. Moreover, we call any line of $\cA$ which is not a side one an \textit{inner line} of the arrangement.
\end{defi}

As we will perceive throughout this work, it will be extremely important to relate two arrangements which differ by one line (and their associated reflexive sheaves). Such a relation is explained by what we will refer to as the \textit{Addition-Deletion theorem} (or simply Addition-Deletion). Even if historically such name is reserved for the result that relates the freeness of the arrangements in a triple (see \cite[Theorem 4.51]{OT} for more details), we will use it to recall the exact sequences we now describe. Nevertheless, the mentioned relations for the freeness are also obtained  using such sequences.

Let $\cA$ be a line arrangement in $\P^2$, and consider a line $l \in \cA$. Denoting by $\cA \backslash l$ the arrangement obtained by removing $l$, we have the  short exact sequence
\begin{equation}\label{seq-deletion}
  \xymatrix@C=16pt{
    0\ar[r]&\T_{\mathcal{A}} \ar[r]&\T_{\mathcal{A}\setminus l}\ar[r]&\cO_l(-t) \ar[r]&0,
    }
\end{equation}
where $t$ denotes the number of triple points in the line $l$ counted with appropriate multiplicity (see \cite[Proposition 5.1]{FV} for more details). This process is known as \textit{deletion}.

By \cite[Proposition 5.2]{FV}, if $\T_{\mathcal{A}} \simeq \cO_{\P^2}(-\alpha)\oplus \cO_{\P^2}(-\beta)$ and $t = \alpha-1$ or $t=\beta-1$, then $\T_{\mathcal{A}\setminus l}$ is free as well, with exponents, respectively, $(\alpha-1,\beta)$ or $(\alpha,\beta-1)$.

Let us now consider  the addition of a line $l$ to an arrangement $\cA$ of $d$ lines. We also have  a short exact sequence
\begin{equation}\label{seq-deletion2}
\xymatrix@C=16pt{
  0\ar[r]&\T_{\mathcal{A}\cup l}\ar[r]&\T_{\mathcal{A}}\ar[r]&\cO_l(-t)\ar[r]&0,
}
\end{equation}
where $t$ denotes the number of triple points of the arrangement $\mathcal{A}\cup l$ in the line $l$ (see again \cite[Proposition 5.1]{FV} for more details).
Considering the dual of the sequence (\ref{seq-deletion2}), we get
$$
\xymatrix@C=16pt{
  0\ar[r]&\T_{\mathcal{A}}(-1)\ar[r]&\T_{\mathcal{A}\cup l}\ar[r]&\cO_l(t+1-d) \ar[r]&0.}
$$
This process is known as \textit{addition}.

Again by \cite[Proposition 5.2]{FV}, if $\T_{\mathcal{A}} \simeq \cO_{\P^2}(-\alpha)\oplus \cO_{\P^2}(-\beta)$ and $t = \alpha$ or $t=\beta$, then $\T_{\mathcal{A}\cup l}$ is free as well, with exponents, respectively, $(\alpha,\beta+1)$ or $(\alpha+1,\beta)$.

\begin{pro}\label{def-TA-T} Let $\mathcal{A}\in \Tr(a,b,c)$; then,
there is an exact sequence 
$$
\xymatrix@C=16pt{
  0\ar[r]& \T_{\mathcal{A}}\ar[r]&\cO_{\P^2}(-a)\oplus \cO_{\P^2}(-b) \oplus \cO_{\P^2}(-c)\ar[r]^-{\psi} &\mathcal{I}_T(-1)\ar[r]&0,}
$$
where $\mathcal{I}_T$ is the ideal sheaf of the smooth finite set $T$ of  \textit{inner} triple points (\textit{i.e.}~$A,B,C\notin T$).

Moreover, if we choose  the points $A=(1:0:0)$, $B=(0:1:0)$ and $C=(0:0:1)$, then $\mathcal{A}\in \Tr(a,b,c)$ is defined by an equation of type
$$
xyz\prod_{i=1}^{a-1}(x-\alpha_i y)\prod_{j=1}^{b-1}(y-\beta_j z)\prod_{k=1}^{c-1}(x-\gamma_k z)=0
$$
with $\alpha_i, \beta_j, \gamma_k \in\C$. Then, after a projective linear transformation if necessary, we have
$$
\psi = \left[\prod_{i=1}^{a-1}(x-\alpha_i y),\prod_{j=1}^{b-1}(y-\beta_j z),\prod_{k=1}^{c-1}(x-\gamma_k z)\right].
$$ 
\end{pro}

\begin{proof}
Let $Z\subset \PD^2$ be the finite set of points corresponding by projective duality to the lines of $\cA$.  Let $\Delta$ be the union of the three lines $L_A,L_B,L_C$ in $\PD^2$ that are dual lines of the three points $A$, $B$ and $C$.  Since $Z$ is contained in $\Delta$, we have the following exact sequence:
$$
\xymatrix@C=16pt{
  0\ar[r]& \cO_{\P^2}(-2) \ar[r]& \mathcal{I}_Z(1) \ar[r]& \mathcal{I}_{Z}/\mathcal{I}_{\Delta}(1)   \ar[r]&0.}
$$
The vertices of $\Delta$ belong to $Z$, which implies that
$$
\mathcal{I}_{Z}/\mathcal{I}_{\Delta}(1) =\cO_{L_A}(-a)\oplus \cO_{L_B}(-b)\oplus \cO_{L_C}(-c).
$$
Let us consider the incidence variety $\F=\{(x,l)\in \P^2\times \PD^2 \mid x\in l\}$ and the projection maps $p\colon \F \rightarrow \P^2$ and $q\colon \F \rightarrow \PD^2$. According to \cite[Theorem 1.3]{FMV}, we have  $\TA=p_*q^*(\mathcal{I}_Z(1))$, and the Fourier--Mukai transform $p_*q^{*}$ applied to the above exact sequence gives
$$
 \xymatrix{
   0\ar[r]&\T_{\mathcal{A}} \ar[r]& \cO_{\P^2}(-a)\oplus \cO_{\P^2}(-b) \oplus \cO_{\P^2}(-c)\ar[r]&    \cO_{\P^2}(-1)\ar[d]\\
   & 0&  \mathrm{R}^1p_*q^{*}\cO_{L_A}(-a)\oplus  \mathrm{R}^1p_*q^{*}\cO_{L_B}(-b)\oplus  \mathrm{R}^1p_*q^{*}\cO_{L_C}(-c) \ar[l]&  \mathrm{R}^1p_*q^{*}\mathcal{I}_Z(1)\rlap{.}\ar[l]}
 $$
    The sheaf $\mathrm{R}^1p_*q^{*}\mathcal{I}_Z(1)$ is supported on the scheme of triple points defined by $\mathcal{A}$, while the last sheaf of the sequence is supported on the vertices of the triangle $(ABC)$. Therefore, the kernel of the last map is the structural sheaf of the set of triple inner points $T$. This implies that we have the following exact sequence: 
    $$
 \xymatrix@C=16pt{
   0\ar[r]& \T_{\mathcal{A}} \ar[r]& \cO_{\P^2}(-a)\oplus \cO_{\P^2}(-b) \oplus \cO_{\P^2}(-c) \ar[r]^-{\psi} &  \mathcal{I}_T(-1) \ar[r]& 0.
   }$$
       
We will now prove, by induction, the explicit description of the map $\psi$.
       
    Let us consider a line arrangement $\cA \in \Tr(2,2,2)$. Then, one can assume that its defining equation is 
    $xyz(x-z)(y-\beta z)(x-\gamma y)=0$. Two different cases appear:
    \begin{itemize}
    \item The three inner lines are not concurrent, \textit{i.e.}~$\beta \gamma\neq 1$.
     \item The three inner lines are  concurrent in a point $p$, \textit{i.e.}~$\beta \gamma=1$.
    \end{itemize}
In the first case, we have 
$$
\xymatrix@C=16pt{
   0\ar[r]& \T_{\mathcal{A}}\ar[r]& \cO_{\P^2}(-2)^3 \ar[r]^-{\psi}&  \cO_{\P^2}(-1)\ar[r]& 0,}
$$
    and therefore $ \T_{\mathcal{A}}\simeq T_{\P^2}(-4)$.
In the second case, we have 
$$
\xymatrix@C=16pt{
   0\ar[r]& \T_{\mathcal{A}} \ar[r]& \cO_{\P^2}(-2)^3 \ar[r]^-{\psi}&  \mathcal{I}_p(-1) \ar[r]&0,}
$$
and therefore $ \T_{\mathcal{A}}\simeq   \cO_{\P^2}(-3) \oplus \cO_{\P^2}(-2)$.

 It is clear in both cases that, after taking a linear combination of the columns of the defining matrix, we have $\psi=[x-z,y-\beta z,x-\gamma y]$.

 Let us now assume  that one can write $\psi$ in the described form for any triple $(a,b,c)$ such that $2\le a \le b \le c$ and $a+b+c\le n$, and let us show the result for a line arrangement $\cA \in \Tr(a,b,c)$ with $a+b+c=n+1$.
 
Let us remove a line $l=\{y=\gamma x\}$ passing through the point $C$, and let us denote by $Z$ the $k$ inner triple points through which the line $l$ passes.  Then this line contains $c+k-1$ triple points counted with multiplicities, more precisely, $k$ inner triple points of $Z$ and the point $C$, at which $c+1$ lines meet, counting for $c-1$ triple points.  By Addition-Deletion, we have an exact sequence
$$
\xymatrix@C=16pt{
  0\ar[r]&\T_{\mathcal{A}}\ar[r]& \T_{\mathcal{A}\setminus l} \ar[r]&  \cO_l(-c-k+1) \ar[r]&0.}
$$  
  This induces a commutative diagram   
  $$
\xymatrix{
  0\ar[r]&\T_{\mathcal{A}}\ar[r]\ar[d]& \cO_{\P^2}(-a)\oplus \cO_{\P^2}(-b) \oplus \cO_{\P^2}(-c)  \ar[r]^-{\psi} \ar[d]_{M}&   \mathcal{I}_T(-1)\ar[r]\ar[d]_{i}& 0\\
  0\ar[r]&\T_{\mathcal{A}\setminus l} \ar[r]\ar[d]& \cO_{\P^2}(-a)\oplus \cO_{\P^2}(-b) \oplus \cO_{\P^2}(-c+1) \ar[r]^-{\psi_0}\ar[d]&  \mathcal{I}_{T\setminus Z}(-1)\ar[r]\ar[d]& 0\\
  0\ar[r]&\cO_l(-c-k+1) \ar[r]& \cO_l(-c+1) \ar[r]& \cO_Z \ar[r]& 0.
  }
  $$
    The morphism $i$ is the inclusion of ideal sheaves, $M$ is represented by the matrix 
    $$
\left(
\begin{array}{ccc}
1&0 &0 \\
0& 1&0\\
0&0&y-\gamma x
\end{array}
\right),
$$
and they verify $i\circ \psi=\psi_0 \circ M$, where $\psi_0$ has the described form by induction. This proves that $\psi$ has the required form.
\end{proof}
\begin{remark}
We would like to underline that the previous result gives us a graded isomorphism between the syzygy modules of the Jacobian ideal $\mathcal{J}_f$ of the arrangement and of the ideal generated by the three products of inner lines of the triangular arrangement (one for each vertex). This relation will be of extreme importance when we will consider the \textit{complementary arrangement}, a concept we will introduce later on to study triangular ones.
\end{remark}

\begin{remark}
\label{chernClass}
 By the hypothesis on $\mathcal{A}$, the set $T$ is smooth. Its length is related to the second Chern class of $\T_{\mathcal{A}}$; more precisely, we have that $c_1(\T_{\mathcal{A}})=1-a-b-c$ and
 \begin{equation}
 c_2(\T_{\mathcal{A}})=\binom{a+b+c-1}{2}-\binom{a}{2}-\binom{b}{2}-\binom{c}{2}-|T| = (ab+bc+ac-a-b-c+1)-|T|.
 \end{equation}
\end{remark}
\begin{theo}\label{CI}
 The bundle $\T_{\mathcal{A}}$ is free with exponents $(a+b-1,c)$ if and only if\, $T$ is a complete intersection $(a-1,b-1)$.
\end{theo}

\begin{remark}\label{bigger-split}
 This is the most unbalanced  splitting that is allowed for $\cA \in \Tr(a,b,c)$, \textit{i.e.}~the highest difference possible between the two exponents. Indeed, $|T|$ cannot be bigger than $(a-1)(b-1)$.
\end{remark}

\begin{proof}
 Assume that $T$ is a complete intersection $(a-1,b-1)$. Since $T$ is the locus of inner triple points, the curve defined by $(a-1)$ lines passing through $A$ contains $T$, and the curve defined by $(b-1)$ lines passing through $B$ also contains  $T$. These two curves generate the ideal defining $T$, which implies that the kernel of the last map of the exact sequence
 $$
\xymatrix@C=16pt{
  0\ar[r]&\T_{\mathcal{A}}\ar[r]& \cO_{\P^2}(-a)\oplus \cO_{\P^2}(-b) \oplus \cO_{\P^2}(-c) \ar[r]&  \mathcal{I}_T(-1) \ar[r]&0}
$$
 is $\cO_{\P^2}(-a-b+1)\oplus \cO_{\P^2}(-c)$. This proves that  $\T_{\mathcal{A}}$ is free with exponents $(a+b-1,c)$.
 
 Reciprocally, if $\T_{\mathcal{A}}=\cO_{\P^2}(-a-b+1)\oplus \cO_{\P^2}(-c)$, then we have $\HH^0(\mathcal{I}_T(a-1))\neq 0$ and $\HH^0(\mathcal{I}_T(a-2))= 0$.  Moreover, the length of $T$, given by the numerical invariant of the above exact sequence, is $(a-1)(b-1)$, and this proves that $T$ is a complete intersection $(a-1,b-1)$.
\end{proof}

\begin{remark}\label{possible-splitting}
  If $c\ge a+b-1$, the splitting type of $\TA$ along the lines joining $A$ to $C$ or $B$ to $C$ is fixed and is $\cO_l(1-a-b)\oplus \cO_l(-c)$; this is a consequence of \cite[Theorem 3.1]{WY}. Therefore, under the condition $c\ge a+b-1$, the arrangement is free if and only if
  $$\T_{\cA}=\cO_{\P^2}(1-a-b)\oplus \cO_{\P^2}(-c).$$
  That is why, if we want to describe all the possible splitting types of free triangular arrangements with $a+b+c$ lines ($a+1$ by $A$, $b+1$ by $B$ and $c+1$ by $C$), we can assume that $c\le a+b-2$. By \cite[Theorem 1.2]{D2}, if $\T_\cA$ is free with exponents $(\alpha, \beta)$, then we have that $c\leq\alpha\leq\beta\leq a+b-1$. This means that the biggest possible gap between the exponents is $| a+b-1-c|$, and, moreover, it is realized by the complete intersection $(a-1)(b-1)$. In particular, it could be described with a \textit{Roots-of-Unity-Arrangement} (see the definition below): let $\zeta$ be a primitive $\supth{(c-1)}$ root of unity; the arrangement
 $$xyz\prod_{i=0}^{a-2}(x-\zeta^i y)\prod_{j=0}^{b-2}(y-\zeta^j z)\prod_{k=0}^{c-2}(x-\zeta^k z)=0$$
 belongs to $\Tr(a,b,c)$, and it is free with exponents $(a+b-1,c)$.
\end{remark}

\begin{defi}
A triangular arrangement $\mathcal{A}$ of $a+b+c$ lines defined by an equation
$$xyz\prod_{i=1}^{a-1}(x-\alpha_i y)\prod_{j=1}^{b-1}(y-\beta_j z)\prod_{k=1}^{c-1}(x-\gamma_k z)=0,$$
is called a \textit{Roots-of-Unity-Arrangement} (RUA for short) if the coefficients $\alpha_i, \beta_j$ and $\gamma_k$ can all be expressed as powers of a fixed $\supth{n}$ root of unity $\zeta$.
\end{defi}

The following results concerning the arrangements $\cA^0_3(n)$, $\cA^1_3(n)$, $\cA^2_3(n)$ and $\cA^3_3(n)$, defined respectively by the equations $f_n=0, xf_n=0, xyf_n=0$ and $xyzf_n=0$, where $f_n= (x^n-y^n)(y^n-z^n)(x^n-z^n)$, are well known (see in particular \cite[Propositions 6.77 and 6.85]{OT} or \cite[p.~119]{Hirzebruch}).  The arrangements $\cA^3_3(n)$ and $\cA^0_3(n)$ are reflection arrangements associated respectively to the full monomial group $G(n,1,3)$ (see \cite[Example 6.29]{OT}) and to the monomial group $G(n,3,3)$; we call them the full monomial arrangement and the monomial arrangement.

\begin{cor}
\label{reflection-3}
  The arrangement $\mathcal{A}_3^3(n)$ is free with exponents $(n+1, 2n+1)$.
\end{cor}

\begin{proof}
The set of inner triple points $T$ is a complete intersection of length $n^2$ defined by the ideal 
$(x^n-y^n,$ $y^n-z^n)$.
\end{proof}

Removing the sides of the triangle and applying \cite[Proposition 5.2]{FV}, we also obtain the following. 

\begin{cor}
\label{reflection-012}
 The arrangements $\mathcal{A}_2^3(n), \mathcal{A}_1^3(n)$ and $\mathcal{A}_0^3(n)$ are obtained, respectively, from $\mathcal{A}_3^3(n)$, $\mathcal{A}_2^3(n)$ and $\mathcal{A}_1^3(n)$ by deleting one line between two vertices of the triangle. They are free with exponents respectively equal to $(n+1, 2n)$, $(n+1, 2n-1)$ and $(n+1, 2n-2)$.
\end{cor}

\section{Roots-of-Unity-Arrangement}\label{sec-RUA}
Let $L(\cA)$ be the set of all the intersections of lines in $\cA$. There is a partial order by reverse inclusion on this set corresponding to the inclusion of points in lines. This setting leads to a very important combinatorial invariant of the arrangement $\cA$.

\begin{defi}
 Let $\cA$ be a line arrangement in $\P^2$. The set $L(\cA)$ is called the {\it combinatorics} of $\cA$.  Two line arrangements $\cA_0$ and $\cA_1$ have the same combinatorics if and only if there is a bijection between $L(\cA_0)$ and $L(\cA_1)$ preserving the partial order.
\end{defi}

 In \cite[Conjecture 4.138]{OT} Terao conjectures that if two arrangements have the same combinatorics and one of them is free, then the other one is also free.  This problem posed in any dimension and on any field is still open even on the projective plane and seems far from being proved, probably because relatively few free arrangements are known.

 In order to approach Terao's conjecture, this section is devoted to showing the central role of RUAs.
 
\begin{theo}\label{thm-roots}
Given a triangular arrangement, there exists a RUA with the same combinatorics.
\end{theo}

\begin{proof}
Let us consider the triangular arrangement defined by the  equations
$$
\left\{
\begin{array}{l}
x=0,\\
x=\alpha_i y,\\
y=0,\\
y = \beta_j z,\\
z=0,\\
z = \gamma_k x, 
\end{array}
\right.
$$
where $x=y=z=0$ are the lines which make up the triangle, $\alpha_i\neq0$, $i=1,\ldots,a-1$, and $\alpha_{i_1}\neq \alpha_{i_2}$ for $i_1\neq i_2$, 
and the same properties hold for the $\beta_j$ and the $\gamma_k$, $j=1,\ldots, b-1$ and $k=1,\ldots,c-1$.

Observe that the existence of an inner triple point, defined by three lines $x=\alpha_{\bar{\imath}} y, y = \beta_{\bar{\jmath}} z$ and $z = \gamma_{\bar{k}} x$, is given by a relation of the  type
$$
\alpha_{\bar{\imath}}\beta_{\bar{\jmath}}\gamma_{\bar{k}}=1.
$$
Therefore, we can translate the combinatorics of the arrangement into a family of equalities
\begin{equation}\label{fam-1}
\alpha_{i_1}\beta_{j_1}\gamma_{k_1}=1
\end{equation}
for each $i_1,j_1,k_1$ whose associated lines define an inner triple point of the arrangement, a family of inequalities
\begin{equation}\label{fam-2}
\alpha_{i_2}\beta_{j_2}\gamma_{k_2}\neq 1
\end{equation}
for each $i_2,j_2,k_2$ whose associated lines do not define an inner triple point of the arrangement and, finally, the inequalities 
\begin{equation}\label{fam-3}
\alpha_{i_1}\neq \alpha_{i_2},\; \alpha_{i_1}\neq 0,\; \beta_{j_1}\neq \beta_{j_2}, \;\beta_{j_1}\neq 0,\; \gamma_{k_1}\neq\gamma_{k_2},\; \gamma_{k_1}\neq 0, 
\end{equation}
for each $i_1,i_2=1,\ldots,a-1$ with $i_1\neq i_2$, each $j_1,j_2=1,\ldots,b-1$ with $j_1\neq j_2$ and each $k_1,k_2=1,\ldots,c-1$ with $k_1\neq k_2$.

Our goal is to find solutions, or at least prove their existence, which satisfy all the previous relations and that can be expressed as various powers of an $\supth{n}$ root of the unity, for a given $n$.

Recall that (see for example \cite{Serre} for a reference) having a solution of the system defined by (\ref{fam-1}) in $\mathbb{C}$ implies having a solution in $\mathbb{Z}/p\mathbb{Z}$ for infinitely many primes $p$. Furthermore, we can consider the equations as in the abelian group of units $(\mathbb{Z}/p\mathbb{Z})^*$, and hence we obtain solutions as $\supth{p}$ roots of unity.

Nevertheless, we would like to show explicitly that these newfound solutions do not satisfy any of the equations described by (\ref{fam-2}) and (\ref{fam-3}); \textit{i.e.}~they continue to satisfy the inequalities.

Let us consider a prime number $p$ and one of its primitive roots $\omega$. Hence, working modulo $p$, we can translate all the relations of type (\ref{fam-1}) to 
$$
\omega^{v_{i_1}}\omega^{w_{j_1}}\omega^{t_{k_1}}\equiv 1\Mod{p}
$$ 
or, equivalently, to a family of linear equations
\begin{equation}\label{fam-4}
v_{i_1}+w_{j_1}+t_{k_1} \equiv 0\Mod{p-1}.
\end{equation}
Now,  add to our previous linear system, modulo $p$, an equality given by one of the equations appearing in (\ref{fam-2}), that we can suppose to be of the form
$$
v+w+t \equiv 0\Mod{p-1}.
$$
Then, either we have fewer solutions than before, or the added condition is a consequence of the others. In particular, suppose that for a fixed $p$, we have exactly the same solutions; this means that the added condition is a linear combination of the previous ones, \textit{i.e.}\ 
$$
v+w+t \equiv \sum_s \lambda_{s,p}\left(v_{i_s}+w_{j_s}+t_{k_s}\right)\Mod{p-1},
$$
which is equivalent to having 
$$
\omega^v \omega^w \omega^t \equiv \prod_s \left(\omega^{v_{i_s}} \omega^{w_{j_s}} \omega^{t_{k_s}}\right)^{\lambda_{s,p}}\Mod{p}
$$
and therefore
$$
\alpha \beta \gamma \equiv \prod_s\left(\alpha_{i_s}\beta_{j_s}\gamma_{k_s}\right)^{\lambda_{s,p}}\Mod{p}.
$$
We conclude by noticing that if we have the previous relation for an infinite set of prime numbers, then we must also have, in the complex field,
$$
\alpha \beta \gamma = \prod_s\left(\alpha_{i_s}\beta_{j_s}\gamma_{k_s}\right)^{\lambda_{s}},
$$
which implies that the added condition is a consequence to the other equalities, which gives a contradiction because of our hypothesis on the triple points.

Following the same reasoning as before for the relations expressed in the family (\ref{fam-3}), it is possible to find solutions that do also satisfy those inequalities.

Therefore, as a consequence of what we have said, it is possible to find infinitely many primes $p$ such that, taking one of its  roots of unity   $\zeta$, we can associate to each coefficient $\alpha_i$, $\beta_j$ and $\gamma_k$, respectively, powers $\zeta^{\bar{\alpha}_i}$, $\zeta^{\bar{\beta}_j}$ and $\zeta^{\bar{\gamma}_k}$ of the root of unity. The arrangement defined by the lines $x = \zeta^{\bar{\alpha}_i}y$, $y=\zeta^{\bar{\beta}_j}z$ and $ z = \zeta^{\bar{\gamma}_k}x$ satisfies all the conditions expressed in (\ref{fam-1}), (\ref{fam-2}) and (\ref{fam-3}) and  therefore has the exact same combinatorics as the original one.
\end{proof}

\section{Characterizations of free Roots-of-Unity-Arrangements}
\label{sec-family}
In this section we describe some free arrangements in $\Tr(a,b,c)$ obtained from a full monomial arrangement; we also give some conditions implying freeness for these arrangements.

First of all, let us introduce some definitions and make some remarks that will be used throughout this section.

Directly from its definition, any Roots-of-Unity-Arrangement $\cA$ can be obtained from a full monomial arrangement by excluding (or, as we will say from now on \textit{deleting}) the lines that are associated to the powers of the $\supth{n}$ root of unity  which do not appear in $\cA$. The set of deleted lines form a further arrangement, that we will call the \textit{complementary arrangement} of $\cA$ and denote by $\cA^c$. We will also call a \textit{step} each deletion of a line in the process that constructs the arrangement $\cA$ from the full monomial one. Observe that the arrangement $\cA^c$ is triangular as well, with no side line of the triangle.  Therefore, it makes sense to define analogously the \textit{triple point}s of the complementary arrangements as the points which belong to the intersections of three lines, one for each family.  To better relate the result and the example with their graphical representation, as we will see later, we will also call \textit{horizontal, vertical}
and \textit{diagonal}  the lines passing through, respectively, $A$ (first family), $B$ (second family) and $C$ (third family).

Let us denote by $a^{'}=N-a+1$, $b^{'}=N-b+1$ and $c^{'}=N-c+1$, respectively, the number of inner lines deleted from a full monomial arrangement $\mathcal{A}_3^3(N)$ and passing through the vertices $A$, $B$ and $C$.  Let us denote by $f_A=0, f_B=0$ and $f_C=0$, respectively, the defining equation for these three families.  Therefore, the complementary arrangement $\cA^c$ will be defined by the equation $f_Af_Bf_C=0$. Finally, denote by $T_{\rem}$ the set of inner triple points of $\cA^c$.
 
We have the  commutative diagram
\begin{equation}\label{diag-syz}
\xymatrix{
 &  \cO_{\P^2}(-1-N) \ar[r]^{\simeq} \ar[d] & \cO_{\P^2}(-1-N)  \ar[d]^{(f_A,f_B,-f_C)^T}\\
 0 \ar[r] & \T_{\cA} \ar[d]  \ar[r]& \cO_{\P^2}(-a)\oplus \cO_{\P^2}(-b)\oplus \cO_{\P^2}(-c) \ar[r]^(.72){\mathrm{\psi}} \ar[d]   & \mathcal{I}_T(-1) \ar[r] \ar[d] & 0 \\
 0\ar[r] & \mathcal{I}_{\Gamma}(N+2-a-b-c)   \ar[r] & \mathcal{F} \ar[r] & \mathcal{I}_{T}(-1) \ar[r] & 0\rlap{,} 
}
\end{equation}
where $\TA$ is the logarithmic bundle associated to the triangular arrangement $\cA\in \Tr(a,b,c)$, and therefore  
$$
\psi=\left[\frac{x^{N}-y^{N}}{f_A},\frac{y^{N}-z^{N}}{f_B} ,\frac{x^{N}-z^{N}}{f_C}\right],
$$
given by the inner triple lines passing, respectively, through each one of the vertices of the triangle (see Proposition~\ref{def-TA-T}). Having that 
$$
\psi \cdot (f_A,f_B,-f_C)^T=0,
$$
we denote by $\Gamma$ the zero locus of the section of $\TA(N+1)$ induced by the syzygy $(f_A,f_B,-f_C)$. Notice that $\Gamma$ cannot have a 1-dimensional component, or else the considered section could be written as the product $g\cdot h$, with $g\in H^0(\cO_{\P^2}(\alpha))$, with $\alpha >0 $, and $h \in H^0(\T_\cA(N+1-\alpha))$. The homogeneous form $g$ would therefore be a common factor of $(f_A,f_B,-f_C)$ as well, which is impossible, as the three families of removed lines have no common one.
This implies that the cokernel of the map $ \cO_{\P^2} \hookrightarrow \T_\cA$ is a rank $1$ torsion-free sheaf, that is, a twist of the ideal sheaf of the vanishing locus $\Gamma$ of the considered map, \textit{i.e.}~of the form $\mathcal{I}_\Gamma(d)$ for some integer $d$.

Moreover, the twist can be determined by computing the first Chern classes in the short exact sequence. Indeed, recall that the Chern polynomials in the sequence are related in the following way: 
$$
c_t(\T_\cA) = c_t(\cO_{\P^2}(-1-N)) \cdot c_t(\mathcal{I}_\Gamma(d)),
$$
which implies that their first Chern classes satisfy the relation
$$
d = c_1(\mathcal{I}_\Gamma(d)) = c_1(\T_\cA) - c_1(\cO_{\P^2}(-1-N)) = N+2-a-b-c.
$$

Consider the \textit{singular locus} of $\mathcal{F}$, denoted by $\sing(\mathcal{F})$, that is defined (for any coherent sheaf in general) as the set of points where $\mathcal{F}$ fails to be locally free. It is known that
$$
\begin{array}{rl}
\sing(\mathcal{F}) & = \left\{ x \in \P^2 \: |\: \mathcal{F} \:\mbox{is not a free}\: \cO_{\P^2,x}-\mbox{module} \right\}\vspace{2mm}\\
& = \bigcup_{p=1}^{\dim 2} \Supp \mathcal{E}xt^p(\mathcal{F},\cO_{\P^2}), 
\end{array}
$$
where $\Supp$ stands for the scheme-theoretic support of the sheaf (see \cite[Lemma 2.1.4.1]{Ok}).

Dualizing the bottom row in Diagram (\ref{diag-syz}), we get
\begin{equation}\label{seq-extsupp}
\xymatrix@C=16pt{
  \mathcal{E}xt^1(\mathcal{I}_{T}(-1),\cO_{\P^2}) \ar[r]& \mathcal{E}xt^1(\mathcal{F},\cO_{\P^2}) \ar[r]&  \mathcal{E}xt^1(\mathcal{I}_{\Gamma}(N+2-a-b-c),\cO_{\P^2}) \ar[r]& 0\rlap{.}}
\end{equation}

As $\Gamma$ is the zero locus of a rank $2$ vector bundle and  $T$ is a reduced 0-dimensional scheme, then 
$$\mathcal{E}xt^1(\mathcal{I}_{\Gamma}(N+2-a-b-c),\cO_{\P^2})=\cO_{\Gamma} \quad\text{and}\quad 
\mathcal{E}xt^1(\mathcal{I}_T(-1),\cO_{\P^2})=\cO_{T}.$$
Furthermore, dualizing the central vertical sequence in Diagram (\ref{diag-syz}), we get
$$
\xymatrix@C=16pt{
  0 \ar[r]& \mathcal{F}^\lor \ar[r]& \cO_{\P^2}(a)\oplus \cO_{\P^2}(b)\oplus \cO_{\P^2}(c) \ar[rr]^-{\;(f_A,f_B,-f_C)\;}&& \cO_{\P^2}(1+N) \ar[r]& \mathcal{E}xt^1(\mathcal{F},\cO_{\P^2}) \ar[r]& 0\rlap{,}}
$$
which implies that $$\mathcal{E}xt^1(\mathcal{F},\cO_{\P^2})=\cO_{T_{\rem}}.$$
This gives a surjective map: $\cO_{T_{\rem}}\to \cO_{\Gamma}. $
Moreover, $\Supp \mathcal{E}xt^1(\mathcal{I}_{T}(-1),\cO_{\P^2}) = T$ and $T\cap T_{\rem}=\emptyset$. This implies that $\Gamma=T_{\rem}$.

Finally, recall that, for an ideal sheaf $\mathcal{I}_\Gamma$, the length of the 0-dimensional scheme $\Gamma$ is given by the second Chern class $c_2(\mathcal{I}_\Gamma(d))$, for any integer $d$. In our case, by the Chern polynomial properties mentioned before, we get that 
 $$|T_{\rem}|=c_2(\mathcal{I}_\Gamma(2N+3-a-b-c))=c_2(\TA(N+1)).$$

Recall that, from Remark~\ref{chernClass}, we have 
\begin{equation}\label{maxChern}
|T| = (ab+bc+ac-a-b-c+1)-c_2(\T_{\mathcal{A}}),
\end{equation}
 from which we obtain the following formula relating 
$|T|$ and $|T_{\rem}|$:
\begin{equation}\label{triple-complementary}
|T| = N^2 - (N+2)(a+b+c-3)-3+ab+ac+bc-|T_{\rem}|.
\end{equation}

\begin{theo}\label{Trem-+}
Let $\cA$ be a triangular arrangement in $\Tr(a,b,c)$ obtained by deletion of a complementary triangular arrangement $\cA^{c}$ from $\mathcal{A}_3^3(N)$. Let $T_{\rem}$ be the inner triple points of $\cA^{c}$. The following holds:
\begin{enumerate}
    \item\label{Trem-+1} If\, $T_{\rem}=\emptyset$, then $\T_\cA$ is free.
\end{enumerate}
Furthermore, 
\begin{enumerate}[resume]
    \item\label{Trem-+2} If\, $2N-a-b-c+2\leq 0$, then $$T_{\rem}=\emptyset \text{ if and only if }\ \T_{\cA} \text{ is free}. $$
    \item\label{Trem-+3} If\, $c\geq a+b-1$, then $$ \T_{\cA} \text{ is free if and only if }\ |T_{\rem}|=(N-c+1)(N-a-b+2).$$ 
    Moreover, in this situation, the number of triple points for each removed line in the third family, which lives in the complementary arrangement, is equal to $N-a-b+2$.
\end{enumerate}
\end{theo}

\begin{remark}\leavevmode
  \begin{enumerate}
  \item If $2N-a-b-c+2> 0$, then $\T_\cA$ can be free even if  $T_{\rem}$ is not empty (see Example~\ref{example-trem} below).
\item In the cases \eqref{Trem-+1} and \eqref{Trem-+2} of Theorem~\ref{Trem-+}, when $\T_\cA$ is free,  its exponents are necessarily $(N+1,a+b+c-N-2)$. In case  \eqref{Trem-+3}, when $\T_\cA$ is free,  its exponents are necessarily $(c,a+b-1)$.
\end{enumerate}
  \end{remark}

\begin{proof}[Proof of Theorem~\ref{Trem-+}]
The first assertion is straightforward from Diagram (\ref{diag-syz}).

For the second assertion, restricting the first vertical exact sequence of Diagram (\ref{diag-syz}) to a line $\ell$, we get
$$
\xymatrix@C=16pt{
0 \ar[r]& \cO_{\ell}(-1-N) \ar[r]&  \TA\otimes \cO_{\ell}\ar[r]&  \mathcal{I}_{T_{\rem}}(N+2-a-b-c) \otimes \cO_{\ell} \ar[r]& 0\rlap{.}}$$
When $\ell$ is general, $\mathcal{I}_{T_{\rem}}(N+2-a-b-c) \otimes \cO_{\ell} \simeq \cO_{\ell}(N+2-a-b-c)$.

If $N+2-a-b-c \le -1-N$, or equivalently $2N-a-b-c+2<0$, this imposes that the generic splitting type is $\cO_{\ell}(-1-N) \oplus \cO_{\ell}(N+2-a-b-c)$. So, under the assumption $2N-a-b-c+2\leq 0$, a line $\ell$ is a jumping line if and only if $T_{\rem}\cap \ell$ is not empty. 
In other words, $\T_\cA$ is free with exponents $(N+1,a+b+c-N-2)$ if and only if $T_{\rem}$ is empty.

To prove the third assertion, we first remark that if $c\geq a+b-1$, then the minimal degree $d$ of a global section of $\T_{\cA}$ is equal to $a+b-1$ (see \cite[Theorem 1.2]{D2}). As a consequence, if $\T_{\cA}$ is free, then it is isomorphic to $\cO_{\P^2}(-c) \oplus \cO_{\P^2}(-a-b+1)$.  Directly from the freeness of $\T_{\cA}$ and the first column of Diagram (\ref{diag-syz}), we have that $T_{\rem}$ is a complete intersection with $|T_{\rem}|=(N-c+1)(N-a-b+2)$. As $N-c+1 \leq N-a-b+2$, the $N-c+1$ lines removed from the third family give one of the two generators of the ideal that defines $T_{\rem}$. By B\'{e}zout's theorem, this implies that $|T_{\rem}\cap \ell|=N-a-b+2$ for any removed line $\ell$ of the third family.

Reciprocally, if $|T_{\rem}|=(N-c+1)(N-a-b+2)$, then Equation (\ref{triple-complementary}) implies that $|T|=(a-1)(b-1)$. This can happen only when $T$ is a complete intersection of type $(a-1,b-1)$ and, by Theorem~\ref{CI}, we get that $\T_{\cA}$ is free.
\end{proof}

As a direct application of Theorem~\ref{Trem-+}, we now provide  a free equilateral RUA (which means that $a=b=c$ and the sides of the triangle $ABC$ belong to the arrangement) for any admissible pair of exponents.

\begin{cor}\label{equilateral}
 Let $n$ be a positive integer.
 For any pair of integers $(d_1,d_2)$ verifying $d_1+d_2=3n-1$ and $n\le d_1\le d_2\le 2n-1$, there exists an arrangement $\cA \in  \Tr(n,n,n)$  which is free with exponents $(d_1,d_2)$.
\end{cor}

\begin{proof}
\looseness=1 The full monomial arrangement $\mathcal{A}_3^3(n-1)$ is free with exponents $(n,2n-1)$. Starting from the arrangement $\mathcal{A}_3^3(n)$ and removing one inner line from each vertex such that $T_{\rem}=\emptyset$, we obtain a new arrangement $\cA_1\in \Tr(n,n,n)$. According to Theorem~\ref{Trem-+}, this arrangement is free with exponents $(n+1, 2n-2)$. In order to find a free arrangement $\cA_k\in \Tr(n,n,n)$ with exponents $(n+k, 2n-k-1)$ for $k\ge 1$ and $2k<n $, we apply the same method: we remove $k$ lines from each vertex of the arrangement $ \mathcal{A}_3^3(n+k-1) $ such that there is no inner triple point in the complementary arrangement (\textit{i.e.}~$T_{\rem}=\emptyset$). This is always possible thanks to the hypothesis $2k< n$, for instance by removing the $3k$ lines with equations 
$$ \prod_{0\le i\le k-1}(x-\zeta^i z)=0, \quad\prod_{0\le j\le k-1}(y-\zeta^j z)=0 \quad\text{and}\quad \prod_{0\le h\le k-1}(x-\zeta^{n+k-2-h} y)=0.$$
The second assertion of Theorem~\ref{Trem-+}, with $N=n+k-1$ and $a=b=c=n$, shows that this arrangement is free with exponents $(n+k, 2n-k-1)$.
\end{proof}

Freeness for RUAs is characterized in Theorem~\ref{Trem-+} when $c\ge a+b-1$.
Let us now study  the case $c<a+b-1$. By \cite[Theorem 1.2]{D2} again, if $c < a+b-1$, then the minimal degree $d$ of a global section of $\T_{\cA}$ satisfies $c\leq d \leq a+b-1$. Moreover, if $d=c$, then $\T_{\cA} \simeq \cO_{\P^2}(-c) \oplus \cO_{\P^2}(-a-b+1)$.

First of all, observe that if $\T_{\cA}$ is free, then it can be expressed as the direct sum
$$
\cO_{\P^2}(-c-k) \oplus \cO_{\P^2}(-a-b+1+k), \quad\text{with}\quad k \leq \frac{a+b-c-1}{2},
$$
which implies that $c+k\leq a+b -1-k$.

\begin{pro}\label{Trem-line}
Let $\cA$ be a RUA such that $c_2(\T_{\cA})=(c+k)(a+b-k-1)$. Then
\begin{enumerate}
    \item\label{Trem-line1} if there exists a line $\ell$ such that $|T_{\rem}\cap \ell| = N-a-b+2+k$ or $|T_{\rem}\cap \ell| = N-c+1-k$, then $\cA$ is free;
    \item\label{Trem-line2} there exists no line $\ell \subset \P^2$ such that 
    $$
    N-a-b+2+k < |T_{\rem}\cap \ell| < N-c+1-k;
    $$
    \item\label{Trem-line3} if there exists a line $\ell$ such that  $|T_{\rem}\cap \ell| > N-c+1-k$, then $\cA$ is not free.
\end{enumerate}
\end{pro}

\begin{proof}
  Let us prove \eqref{Trem-line1}. Recall that, by \cite{EF}, if $\mathcal{E}$ is a vector bundle on $\P^2$ and $\mathcal{E}_{|\ell} \simeq \cO_\ell(-\alpha) \oplus \cO_\ell(-\beta)$ for any line $\ell$, then $c_2(\mathcal{E})\geq \alpha\beta$ and, moreover, equality holds if and only if $\mathcal{E} \simeq \cO_{\P^2}(-\alpha) \oplus \cO_{\P^2}(-\beta)$.
  
Suppose that $|T_{\rem}\cap \ell| = N-a-b+2+k$, which, by restricting to $\ell$ the left vertical exact sequence in Diagram~(\ref{diag-syz}), gives us the short exact sequence
$$\xymatrix@C=16pt{
0 \ar[r]& \cO_{\ell}(-N-1) \ar[r]& (T_{\cA})_{|\ell}\ar[r]&  \cO_{\ell}(-c-k) \oplus \cO_{T_{\rem}\cap \ell} \ar[r]& 0\rlap{.}}
$$
This implies that $(T_{\cA})_{|\ell} \simeq \cO_\ell(-\alpha) \oplus \cO_\ell(-\beta)$ with $c+k \leq \alpha \leq \beta \leq a+b+1-k$ and $\alpha + \beta = a+b+c-1$. If $\alpha > c+k$, we would have that $c_2(\mathcal{E})< \alpha\beta$, contradicting the cited result.

The case $|T_{\rem}\cap \ell| = N-c+1-k$ and the other items are proven by analogous arguments.
\end{proof}

\begin{exa}      \label{example-trem}  
We will now see three examples which describe the following different situations:\begin{enumerate}[label={Case \arabic*:}, ref={\arabic*}, leftmargin=!, align=left]
\item\label{exa-trem-1} The complementary arrangement has no inner triple point, and therefore $\T_\cA$ is free.
\item\label{exa-trem-2} The complementary arrangement has inner triple points, and $\T_\cA$ is not free.
\item\label{exa-trem-3} The complementary arrangement has inner triple points, and $\T_\cA$ is free.
  \end{enumerate}

In order to make such example more clear, we will picture the triangular arrangements involved in them as follows. The vertical, horizontal and diagonal lines represent the inner lines of a complete triangular arrangement; \textit{i.e.}~each family has $N$ lines, as $N$ the order of the root of the unity that we are considering. We do not draw the three lines which are the sides of the triangle. Notice that for the horizontal and vertical lines, each drawn line corresponds to a unique line of the arrangement. On the contrary, except for the main diagonal, the diagonal lines are ``paired'' because of the periodicity of the exponential function on the complex field. Indeed, the $\supth{i}$ diagonal line, counted starting from the top left corner, for $1\leq i \leq N$, is paired with the $\supth{(i+N)}$ line of the representation. The lines removed  to obtain the triangular arrangement from the complete one will be the dashed ones.

\subsubsection*{\it Case~\ref{exa-trem-1}: $a=b=c=5$ and $N=6$ ($a'=b'=c'=2$)} 
There is no inner triple point in the complementary arrangement, the arrangement $\cA$ is free with exponents $(7,7)$, and  $|T|=12$.

\begin{figure}[h!]
  \label{trem1}
  \begin{minipage}{0.48\textwidth}
\centering
\begin{tikzpicture}[line cap=round,line join=round,>=triangle 45,x=0.5cm,y=0.5cm]
\definecolor{uququq}{rgb}{0.25,0.25,0.25}
\definecolor{qqqqff}{rgb}{0,0,1}
\clip(-1,-1.2) rectangle (6,5.8);
\draw [domain=-1:6] plot(\x,{(--1-0*\x)/1});
\draw [domain=-1:6] plot(\x,{(--2-0*\x)/1});
\draw (5,-1.2) -- (5,5.8);
\draw [domain=-1:6] plot(\x,{(--3-0*\x)/1});
\draw (2,-1.2) -- (2,5.8);
\draw (3,-1.2) -- (3,5.8);
\draw (4,-1.2) -- (4,5.8);
\draw [line width=1.6pt,dash pattern=on 4pt off 4pt] (1,-1.2) -- (1,5.8);
\draw [line width=1.6pt,dash pattern=on 4pt off 4pt,domain=-1:6] plot(\x,{(--4-0*\x)/1});
\draw [line width=1.6pt,dash pattern=on 4pt off 4pt] (0,-1.2) -- (0,5.8);
\draw [domain=-1:6] plot(\x,{(-0-0*\x)/1});
\draw [domain=-1:6] plot(\x,{(-10-0*\x)/-5});
\draw [line width=1.6pt,dash pattern=on 4pt off 4pt,domain=-1:6] plot(\x,{(--5-0*\x)/1});
\draw [line width=1.6pt,dash pattern=on 4pt off 4pt,domain=-1:6] plot(\x,{(-0--5*\x)/5});
\draw [line width=1.6pt,dash pattern=on 4pt off 4pt,domain=-1:6] plot(\x,{(--4--4*\x)/4});
\draw [line width=1.6pt,dash pattern=on 4pt off 4pt,domain=-1:6] plot(\x,{(-10.3--2.04*\x)/1.86});
\begin{scriptsize}
\fill [color=uququq] (2,0) circle (1.5pt);
\fill [color=uququq] (3,1) circle (1.5pt);
\fill [color=uququq] (3,0) circle (1.5pt);
\fill [color=uququq] (4,2) circle (1.5pt);
\fill [color=uququq] (4,1) circle (1.5pt);
\fill [color=uququq] (4,0) circle (1.5pt);
\fill [color=uququq] (5,1) circle (1.5pt);
\fill [color=uququq] (5,2) circle (1.5pt);
\fill [color=uququq] (2,1) circle (1.5pt);
\fill [color=uququq] (3,2) circle (1.5pt);
\fill [color=uququq] (4,3) circle (1.5pt);
\fill [color=uququq] (5,3) circle (1.5pt);
\end{scriptsize}
\end{tikzpicture}
\caption{Case~\ref{exa-trem-1}: $|T|=12$, $|T_{\rem}|=0$}
  \end{minipage}\hfill
   \begin{minipage}{0.48\textwidth}
\label{trem2}
\centering
\definecolor{uququq}{rgb}{0.25,0.25,0.25}
\definecolor{ffqqtt}{rgb}{1,0,0.2}
\definecolor{xdxdff}{rgb}{0.49,0.49,1}
\begin{tikzpicture}[line cap=round,line join=round,>=triangle 45,x=0.3cm,y=0.3cm]
\clip(-1.7,-1.76) rectangle (8.66,8.86);
\draw [line width=1.6pt,dash pattern=on 2pt off 2pt] (0,-1.76) -- (0,8.86);
\draw [line width=1.6pt,dash pattern=on 2pt off 2pt] (1,-1.76) -- (1,8.86);
\draw [line width=1.6pt,dash pattern=on 2pt off 2pt] (2,-1.76) -- (2,8.86);
\draw (4,-1.76) -- (4,8.86);
\draw (4,-1.76) -- (4,8.86);
\draw [line width=1.6pt,dash pattern=on 2pt off 2pt] (3,-1.76) -- (3,8.86);
\draw [domain=-1.7:8.66] plot(\x,{(-0-0*\x)/1});
\draw [domain=-1.7:8.66] plot(\x,{(--2-0*\x)/1});
\draw [line width=1.6pt,dash pattern=on 2pt off 2pt,domain=-1.7:8.66] plot(\x,{(--4-0*\x)/1});
\draw [line width=1.6pt,dash pattern=on 2pt off 2pt,domain=-1.7:8.66] plot(\x,{(--5-0*\x)/1});
\draw [line width=1.6pt,dash pattern=on 2pt off 2pt,domain=-1.7:8.66] plot(\x,{(--6-0*\x)/1});
\draw [line width=1.6pt,dash pattern=on 2pt off 2pt,domain=-1.7:8.66] plot(\x,{(--7-0*\x)/1});
\draw [domain=-1.7:8.66] plot(\x,{(--1-0*\x)/1});
\draw [domain=-1.7:8.66] plot(\x,{(--2-0*\x)/1});
\draw [domain=-1.7:8.66] plot(\x,{(--3-0*\x)/1});
\draw (5,-1.76) -- (5,8.86);
\draw (6,-1.76) -- (6,8.86);
\draw (7,-1.76) -- (7,8.86);
\draw [line width=1.6pt,dash pattern=on 2pt off 2pt,domain=-1.7:8.66] plot(\x,{(-0-1*\x)/-1});
\draw [line width=1.6pt,dash pattern=on 2pt off 2pt,domain=-1.7:8.66] plot(\x,{(-1-1*\x)/-1});
\draw [line width=1.6pt,dash pattern=on 2pt off 2pt,domain=-1.7:8.66] plot(\x,{(-2-1*\x)/-1});
\draw [line width=1.6pt,dash pattern=on 2pt off 2pt,domain=-1.7:8.66] plot(\x,{(--1-1*\x)/-1});
\draw [line width=1.6pt,dash pattern=on 2pt off 2pt,domain=-1.7:8.66] plot(\x,{(--6-1*\x)/-1});
\draw [line width=1.6pt,dash pattern=on 2pt off 2pt,domain=-1.7:8.66] plot(\x,{(--7-1*\x)/-1});
\draw [line width=1.6pt,dash pattern=on 2pt off 2pt,domain=-1.7:8.66] plot(\x,{(-7-1*\x)/-1});
\begin{scriptsize}
\fill [color=uququq] (4,0) circle (1.5pt);
\fill [color=uququq] (4,1) circle (1.5pt);
\fill [color=uququq] (4,2) circle (1.5pt);
\fill [color=uququq] (5,3) circle (1.5pt);
\fill [color=uququq] (5,2) circle (1.5pt);
\fill [color=uququq](5,1) circle (1.5pt);
\fill [color=uququq](5,0) circle (1.5pt);
\fill [color=uququq](6,0.98) circle (1.5pt);
\fill [color=uququq](5.96,2) circle (1.5pt);
\fill [color=uququq](6,3.04) circle (1.5pt);
\fill [color=uququq](7,2.08) circle (1.5pt);
\fill [color=uququq](7,3.08) circle (1.5pt);
\draw[color=black] (3.16,10.7) node {$l$};
\draw[color=black] (6.88,10.66) node {$g_1$};
\fill [color=ffqqtt] (0,7) circle (2.5pt);
\fill [color=ffqqtt] (3,4) circle (2.5pt);
\fill [color=ffqqtt] (2,4) circle (2.5pt);
\fill [color=ffqqtt] (3,5) circle (2.5pt);
\end{scriptsize}
\end{tikzpicture}
\caption{Case~\ref{exa-trem-2}: $|T|=12, |T_{\rem}|=4$}
\end{minipage}
   \end{figure}

\subsubsection*{\it Case~\ref{exa-trem-2}: $a=b=c=5$ and $N=8$ ($a'=b'=c'=4$)} The number of inner triple points of the complementary arrangement
is minimal, and it is equal to 
$$|T_{\rem}|=0+1+1+2=4.$$ The arrangement $\cA$ is free with exponents $(7,7)$, and $|T|=12$.

\subsubsection*{\it Case~\ref{exa-trem-3}: $a=3, b=4, c=5$  and $N=7$ ($a'=5, b'=4, c'=3$)} The number of inner triple points of the complementary arrangement
is minimal, and it is equal to
$$|T_{\rem}|=2+2+2=6.$$ The arrangement $\cA$ is free with exponents $(5,6)$, and $|T|=6$.
\begin{figure}[h!]
  \label{trem3}
\centering
\definecolor{uququq}{rgb}{0.25,0.25,0.25}
\definecolor{qqqqff}{rgb}{0,0,1}
\definecolor{xdxdff}{rgb}{0.49,0.49,1}
\definecolor{fftttt}{rgb}{1,0.2,0.2}
\begin{tikzpicture}[line cap=round,line join=round,>=triangle 45,x=0.5cm,y=0.5cm]
\clip(-1.3,-3.1) rectangle (6.96,5.1);
\draw [domain=-1.3:6.96] plot(\x,{(-1-0*\x)/1});
\draw [domain=-1.3:6.96] plot(\x,{(-2-0*\x)/1});
\draw [line width=1.2pt,dash pattern=on 2pt off 2pt,domain=-1.3:6.96] plot(\x,{(-12.08--6*\x)/5.98});
\draw [line width=1.2pt,dash pattern=on 2pt off 2pt,domain=-1.3:6.96] plot(\x,{(-6.04--2.04*\x)/2.04});
\draw [line width=1.2pt,dash pattern=on 2pt off 2pt,domain=-1.3:6.96] plot(\x,{(--24.16--6*\x)/5.98});
\draw [line width=1.2pt,dash pattern=on 2pt off 2pt,domain=-1.3:6.96] plot(\x,{(-0-0*\x)/1});
\draw [line width=1.2pt,dash pattern=on 2pt off 2pt,domain=-1.3:6.96] plot(\x,{(--1-0*\x)/1});
\draw [line width=1.2pt,dash pattern=on 2pt off 2pt,domain=-1.3:6.96] plot(\x,{(--2-0*\x)/1});
\draw [line width=1.2pt,dash pattern=on 2pt off 2pt,domain=-1.3:6.96] plot(\x,{(--3-0*\x)/1});
\draw [line width=1.2pt,dash pattern=on 2pt off 2pt,domain=-1.3:6.96] plot(\x,{(--4-0*\x)/1});
\draw [line width=1.2pt,dash pattern=on 2pt off 2pt] (0,-3.1) -- (0,5.1);
\draw [line width=1.2pt,dash pattern=on 2pt off 2pt] (1,-3.1) -- (1,5.1);
\draw [line width=1.2pt,dash pattern=on 2pt off 2pt] (2,-3.1) -- (2,5.1);
\draw [line width=1.2pt,dash pattern=on 2pt off 2pt] (3,-3.1) -- (3,5.1);
\draw (4,-3.1) -- (4,5.1);
\draw (5,-3.1) -- (5,5.1);
\draw (6,-3.1) -- (6,5.1);
\draw [line width=1.2pt,dash pattern=on 2pt off 2pt,domain=-1.3:6.96] plot(\x,{(-16--4*\x)/4});
\draw [line width=1.2pt,dash pattern=on 2pt off 2pt,domain=-1.3:6.96] plot(\x,{(--3--1*\x)/1});
\begin{scriptsize}
\fill [color=fftttt] (2.06,0.04) circle (2.0pt);
\fill [color=fftttt] (3.02,1.01) circle (2.0pt);
\fill [color=fftttt] (3,0.04) circle (2.0pt);
\fill [color=fftttt] (0,4.04) circle (2.0pt);
\fill [color=uququq](3.98,-1) circle (1.5pt);
\fill [color=uququq](4,-2) circle (1.5pt);
\fill [color=uququq](5,-2) circle (1.5pt);
\fill [color=uququq](5,-1) circle (1.5pt);
\fill [color=uququq](5.98,-1) circle (1.5pt);
\fill [color=uququq](5.96,-2) circle (1.5pt);
\fill [color=qqqqff] (8.48,-8.56) circle (1.5pt);
\draw[color=qqqqff] (8.64,-8.3) node {$R$};
\fill [color=fftttt] (0,3) circle (2.0pt);
\fill [color=fftttt] (1,4) circle (2.0pt);
\end{scriptsize}
\end{tikzpicture}
\caption{Case~\ref{exa-trem-3}: $|T|=|T_{\rem}|=6$}
\end{figure}
\end{exa}

\section{Free arrangements in the non-complete triangle}\label{sec-novertex}

In the previous part of this work, we have always considered arrangements coming from full monomial ones by deletion, and, as specified, they all contain the three sides of the triangle.

In this section we will complete the study of triangular arrangements, proving that the only free non-complete triangular arrangements always have the maximal number of inner triple points. More precisely, we have the following theorem.

\begin{theo}
 \label{uncomplete} Let $\cA_0 \in \Tr(a,b,c)$, and let $(AB),(AC), (BC)$ be its three sides. 
 \begin{enumerate}
 \item If we remove one side, then: 
 \begin{itemize}
  \item $\cA_0\setminus (AB)$ is free if and only if 
 its set of inner triple points is a complete intersection $(a-1,b-1)$. Then its exponents are $(c,a+b-2)$.
 \item $\cA_0\setminus (AC)$ is free if and only if $b=c$ and
 its set of inner triple points is a complete intersection $(a-1,b-1)$. Then its exponents are $(b,a+b-2)$.
 \item $\cA_0\setminus (BC)$ is free if and only if $a=b=c$ and
 its set of inner triple points is a complete intersection $(a-1,a-1)$. Then its exponents are $(a,2a-2)$.
 \end{itemize}
  \item If we remove two sides, then: 
 \begin{itemize}
 \item $\cA_0\setminus [(AB)\cup (AC)]$ is free if and only if 
 its set of inner triple points is a complete intersection $(a-1,b-1)$. Then its exponents are $(c,a+b-3)$.
 \item $\cA_0\setminus [(AB)\cup (BC)]$ is free if and only if $b=c$ and
 its set of inner triple points is a complete intersection $(a-1,b-1)$. Then its exponents are $(b,a+b-3)$.
 \item $\cA_0\setminus [(AC)\cup (BC)]$ is free if and only if $a=b=c$ and
 its set of inner triple points is a complete intersection $(a-1,a-1)$. Then its exponents are $(a,2a-3)$.
 \end{itemize}
 \item If we remove three sides, then $\cA_0\setminus [(AB)\cup (AC)\cup (BC)]$ is free if and only if $a=b=c$ and
 its set of inner triple points is a complete intersection $(a-1,a-1)$. Then its exponents are $(a,2a-4)$, or we have $a=b=c=2$ with no triple point and its exponents are $(1,1)$. 
\end{enumerate}
\end{theo}

\begin{proof} We divide the proof in three parts:  $\cA$ is a triangular arrangement with one side removed (Part~\ref{part1}),   two sides removed (Part~\ref{part2}) or three sides removed (Part~\ref{part3})
from a complete triangular arrangement  $\cA_0\in \Tr(a,b,c)$.
\begin{enumerate}[wide, label=\Roman*., ref=\Roman*]
\item\label{part1}  Let us first assume  that $\cA$ contains exactly two sides, and let $L$ be the missing one. Then $\cA\cup L=\cA_0 \in \Tr(a,b,c)$. It is clear that 
 $\cA$ and $\cA\cup L$ have the same set $T$ of inner triple points.
 Let us denote by $t_{\cA}$ and  $t_{\cA \cup L}$, respectively, the number of triple points of $\cA$ and $\cA \cup L$ counted with multiplicities. Then we have 
\begin{itemize}
\item $t_{\cA \cup L}=|T|+\binom{a}{2}+\binom{b}{2}+\binom{c}{2}$;
 \item $t_{\cA}=|T|+\binom{a}{2}+\binom{b-1}{2}+\binom{c-1}{2}$ if $L=(BC)$; 
 \item $t_{\cA}=|T|+\binom{a-1}{2}+\binom{b}{2}+\binom{c-1}{2}$ if $L=(AC)$; 
 \item $t_{\cA}=|T|+\binom{a-1}{2}+\binom{b-1}{2}+\binom{c}{2}$ if $L=(AB)$.
\end{itemize}
The difference $(t_{\cA \cup L}-t_{\cA})$ is $b+c-2$, $a+c-2$ or $a+b-2$, respectively, when $L=(BC), (AC)$ or $(AB)$.

First assume that $L=(AB)$. 
By Addition-Deletion, we have an exact sequence
$$\xymatrix@C=16pt{
0\ar[r]&\T_{\cA \cup L} \ar[r]& \T_{\cA}\ar[r]&  \cO_L(2-a-b) \ar[r]&0.}
$$ 

If $c\le a+b-2$, the surjection of $\T_{\cA}$  on $\cO_L(2-a-b) $ induces
$(\T_{\cA})_{\mid L}=\cO_L(-c)\oplus \cO_L(2-a-b)$ and, if we suppose $\cA$ to be free, we have  $\TA=\cO_{\P^2}(-c)\oplus \cO_{\P^2}(2-a-b)$. Again by Addition-Deletion, we have
 $\mathcal{T}_{\cA \cup L}=\cO_{\P^2}(-c)\oplus \cO_{\P^2}(1-a-b)$ and, according to  Theorem~\ref{CI}, the set of inner triple points is a complete intersection $(a-1,b-1)$.
 
If $c\ge a+b-1$, we consider the multi-arrangement given by the restriction of $\cA$ to the side $(CB)$. The multiplicities of this multi-arrangement
 are $(c,1,\ldots, 1,b-1)$ (with $a-1$ multiplicities of type ``$1$''). This implies that the splitting type  of $\mathcal{T}_{\cA}$ on the side $l=(BC)$ 
is $\cO_l(-c)\oplus \cO_l(2-a-b)$ by \cite[Theorem 3.1]{WY}. 
 Then  $\cA$ is free if and only if $\TA=\cO_{\P^2}(-c)\oplus \cO_{\P^2}(2-a-b)$. By Addition-Deletion again,
$\mathcal{T}_{\cA \cup L}=\cO_{\P^2}(-c)\oplus \cO_{\P^2}(1-a-b)$, and  the set of inner triple points is a complete intersection $(a-1,b-1)$.

 When  $L=(AC)$ or $L=(BC)$, the proofs are analogous. The additional equalities $b=c$ or $a=b=c$ come from the fact that otherwise, by permuting the roles of $a,b,c$, the gaps $|a+c-1-b|$ or $|b+c-1-a|$ would be strictly bigger than the maximal one, which is $|a+b-1-c|$ by Remark~\ref{bigger-split}.

\item\label{part2} Let us now assume that $\cA$ contains only one side, and let $L$ and $L^{'}$ be the missing sides. Then $\cA\cup L\cup L^{'}=\cA_0 \in \Tr(a,b,c)$.  Let $t_{\cA}$, $t_{\cA \cup L}$ and $t_{\cA \cup L\cup L^{'}}$, respectively, be the number of triple points of $\cA$, $\cA \cup L$ and $\cA \cup L\cup L^{'}$, respectively, counted with multiplicities. Then we have
\begin{itemize}
\item $t_{\cA \cup L\cup L^{'}}=|T|+\binom{a}{2}+\binom{b}{2}+\binom{c}{2}$;
 \item $t_{\cA }=|T|+\binom{a-2}{2}+\binom{b-1}{2}+\binom{c-1}{2}$ if $L=(AB)$ and $L^{'}=(AC)$; 
 \item $t_{\cA}=|T|+\binom{a-1}{2}+\binom{b-2}{2}+\binom{c-1}{2}$ if $L=(BC)$ and $L^{'}=(AB)$; 
 \item $t_{\cA}=|T|+\binom{a-1}{2}+\binom{b-1}{2}+\binom{c-2}{2}$ if $L=(AC)$ and $L^{'}=(BC)$. 
\end{itemize}
The difference $(t_{\cA \cup L}-t_{\cA})$ is equal to $a+b-3$ when  $L=(AB)$ and $L^{'}=(AC)$, $a+c-3$ when $L=(AC)$ and $L^{'}=(BC)$, and it is $b+c-3$ when $L=(BC)$ and $L^{'}=(AB)$.

First assume that  $L=(AB)$ and $L^{'}=(AC)$. 
By \cite[Proposition 5.1]{FV} we have an exact sequence
$$\xymatrix@C=16pt{
  0\ar[r]&\T_{\cA \cup L} \ar[r]& \T_{\cA}\ar[r]&  \cO_L(3-a-b) \ar[r]&0.}
$$ 

If $c\le a+b-3$, the surjection of $\T_{\cA}$  on $\cO_L(3-a-b) $ induces
$(\T_{\cA})_{\mid L}=\cO_L(-c)\oplus \cO_L(3-a-b)$, and if we suppose $\cA$ to be free, it gives $\TA=\cO_{\P^2}(-c)\oplus \cO_{\P^2}(3-a-b)$. By Addition-Deletion,
$\cA \cup L$ is free with exponents $(c, a+b-2)$, and, according to Part~\ref{part1},  the set of inner triple points is a complete intersection of type $(a-1,b-1)$.
 
If $c\ge a+b-2$, we consider the other exact sequence given by Addition-Deletion, \textit{i.e.}
\begin{equation}
 \label{trois-a-c}
\xymatrix@C=16pt{
  0\ar[r]&\T_{\cA \cup L^{'}} \ar[r]& \T_{\cA}\ar[r]&  \cO_{L^{'}}(3-a-c)\ar[r]& 0.}
\end{equation}  

If $a=2$, considering the multiplicities of the multi-arrangement obtained by restricting $\cA$ onto the unique line through $A$, we find $|T|$ double points and $c+b-1-2|T|$ simple points.  Notice that since $|T|\le b-1$, the inequality
$$
2(c+b-1-|T|)-1 \geq b+c
$$
is always satisfied, except when $|T|=b-1$ and $b=c$. If these last two conditions hold, $|T|$ is a complete intersection of type $(1,b-1)$, and the corresponding arrangement is free with exponents $(b,b-1)$ (see \cite[Section 2.2]{AFV} for instance).

Let us now suppose that either $|T|<b-1$ or $b<c$. If $\cA$ is free, then it must have exponents $(1+ |T|,$  $c+b-2-|T|)$ by \cite[Theorem 3.1]{WY}.

If $b<c$ but $|T|=b-1$, then $|T|$ is again a complete intersection, and relating as before the bundle $\T_{\cA}$ to the vector bundle $\T_{\cA \cup L \cup L'}$ through Addition-Deletion, we get the required exponents.

Let us now suppose that $|T|<b-1$; it implies $-1-|T| \geq 1-c$ because $b\le c$. If we have a strict inequality, then we have that $2-c-b+|T| = 1-c$ thanks to the surjective map (\ref{trois-a-c}), which implies that $|T|=b-1$ and hence a contradiction. If the equality holds, then we must have that $|T|=b-2$ and $b=c$, which means that if $\cA$ is free, then $\T_{\cA} = \cO_{\P^2}(-b) \oplus \cO_{\P^2}(1-b)$.  This would imply, using Addition-Deletion, that $\T_{\cA \cup L'} = \cO_{\P^2}(-b) \oplus \cO_{\P^2}(-b)$, which leads to a contradiction to Part~\ref{part1} because $T$ is not a complete intersection $(1,b-1)$.

Let us now assume  that $a\ge 3$.

Since $a\ge 3$, it is clear that $b\le a+c-3$ and the surjection of $\T_{\cA}$ on $\cO_{L^{'}}(3-a-c) $ induces $(\T_{\cA})_{\mid L^{'}}=\cO_{L^{'}}(-b)\oplus \cO_{L^{'}}(3-a-c)$, and if we suppose $\cA$ to be free, we obtain $\TA=\cO_{\P^2}(-b)\oplus \cO_{\P^2}(3-a-c)$.  Then, by Addition-Deletion, ${\cA \cup L^{'}}$ is free with exponents $(b, a+c-2)$. According to Part~\ref{part1}, this occurs if and only if $b=c$ and $T$ is a complete intersection $(a-1,b-1)$.

If we remove the lines $(AC)$ and $(BC)$ or the lines $(AB)$ and $(BC)$, the proofs are analogous. The additional equalities $b=c$ or $a=b=c$ come from the fact that otherwise, by permuting the roles of $a,b,c$, the gaps $|a+c-2-b|$ or $|b+c-2-a|$ would be strictly bigger than the maximal one, which is $|a+b-2-c|$ by Remark~\ref{bigger-split}.

\item\label{part3}
Let us now assume  that  $\cA$ contains no side of the triangle, and let $L=(AB)$, $L^{'}=(AC)$ and $L^{''}=(BC)$ denote the missing sides. Then 
$\cA\cup L\cup L^{'} \cup L^{''}=\cA_0 \in \Tr(a,b,c)$. 
Consider the exact sequence 
$$\xymatrix@C=16pt{
  0\ar[r]& \T_{\cA \cup L^{''}} \ar[r]& \T_{\cA}\ar[r]&  \cO_{L^{''}}(4-b-c)\ar[r]& 0.}
$$

If $a\le b+c-4$, then, by the same technique used before, $ \cA$ is free with exponents $(a,b+c-4)$, implying that $\cA \cup L^{''}$ is free with exponents $(a,b+c-3)$. Since the gap for this last splitting is bigger than the gap given by the exponents $(c,a+b-3)$, this proves, according to Part~\ref{part2}, that $\cA$ is free if and only if $a=b=c$ and $T$ is a complete intersection $(a-1,a-1)$.
 
If $a \ge b+c-3$, the possibilities are $(a,b,c)=(2,2,2)$, $(a,b,c)=(2,2,3)$ or $(a,b,c)=(3,3,3)$.The first case, corresponding to three lines, is free with exponents $(1,1)$ if there is no triple point (but this case corresponds to an arrangement in $\Tr(1,1,1)$), and with exponents $(0,2)$ if there is one triple point.  The second case, corresponding to four lines, is free with exponents $(1,2)$ if and only if there is one triple point, \textit{i.e.}~$T$ is a complete intersection $(1,1)$; it cannot be free with exponents $(0,3)$ because the $|T|$ would be $3$, which is impossible for this arrangement.  The third case, corresponding to six lines, is free with exponents $(2,3)$ if and only if $|T|=4$ is a complete intersection $(2,2)$; it cannot be free with exponents $(1,4)$ because $|T|$ would be $6$, which is impossible for this arrangement.
\qedhere
\end{enumerate}
\end{proof}

As a direct consequence of the previous result, we can make explicit the following cases, which describe when we have the same number of lines passing through each vertex.

\begin{cor} Up to a linear change of coordinates, the following hold: 
\begin{itemize}
 \item The only free triangular arrangement of\, $3n-1$ lines passing through three points, with $n-1$ inner lines through each vertex, plus two sides,
 is $\cA^2_3(n)$. 
 \item The only free triangular arrangement of\, $3n-2$ lines passing through three points,  with $n-1$ inner lines through each vertex, plus one side,
 is $\cA^1_3(n)$.
 \item The only free arrangement of\, $3n-3$ lines passing through three points, with $n-1$ inner lines through each vertex, and no side,
 is $\cA^0_3(n)$.
\end{itemize}
\end{cor}

\section{Weak combinatorics}\label{sec-weak}
Being given an arrangement, one can also consider  its  {\it weak combinatorics}.

\begin{defi}
 Let $\cA$ be an arrangement of $n$ lines in $\P^2$. Its weak combinatorics is the data of integers $t_i$ for $ 2\le i \le n$, where $t_i$ is the number of points that are contained in exactly $i$ lines of the arrangement $\cA$.
 \end{defi}

In this section we will show that if we only assume the weak combinatorics hypothesis, the conjecture does not hold. Indeed, we get the following result.

\begin{theo}
 There exist pairs of arrangements possessing the same weak combinatorics such that one is free and the other is not.
\end{theo}

\begin{proof}
 We prove it by describing an example. We will explain next how to produce a family of  examples of the same kind.

We construct two triangular arrangements:  $\cA_0$ which is free with exponents $(7,7)$  and $\cA_1$ which is not free,  both in $\Tr(5,5,5)$ 
with the same numbers of multiple points 
$t_3=12, t_4=t_5=0$, $t_6=3$ and $t_i=0$ for $i>6$ (the number of double points is  given by the combinatorial formula $\binom{15}{2}=t_2+3t_3+\binom{6}{2}t_6$).

\subsubsection*{\it Construction of $\cA_0$} It is obtained by removing the six lines  $x=z$, $x=\zeta z$, $y=z$, $y=\zeta z$, $x=\zeta^2 y$ and $x=\zeta^4 y$
from the full monomial arrangement $xyz(x^6-y^6)(y^6-z^6)(x^6-z^6)=0$ as
represented in the following picture: 

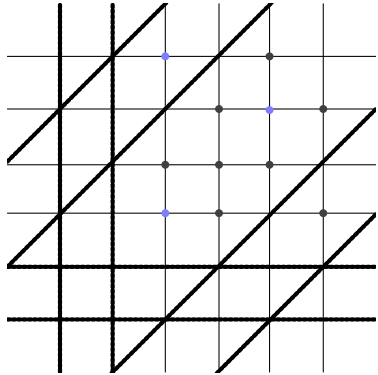
\begin{center}

\definecolor{uququq}{rgb}{0.25,0.25,0.25}
\definecolor{xdxdff}{rgb}{0.49,0.49,1}
\begin{tikzpicture}[line cap=round,line join=round,>=triangle 45,x=0.7cm,y=0.7cm]
\clip(-1,-1) rectangle (6,6);
\draw [line width=1.6pt,dash pattern=on 3pt off 3pt] (0,-1) -- (0,6);
\draw [line width=1.6pt,dash pattern=on 3pt off 3pt] (1,-1) -- (1,6);
\draw (2,-1) -- (2,6);
\draw (3.02,-1) -- (3.02,6);
\draw (3.98,-1) -- (3.98,6);
\draw (5,-1) -- (5,6);
\draw [line width=1.6pt,dash pattern=on 3pt off 3pt,domain=-1:6] plot(\x,{(-0-0*\x)/1});
\draw [line width=1.6pt,dash pattern=on 3pt off 3pt,domain=-1:6] plot(\x,{(--1-0*\x)/1});
\draw [domain=-1:6] plot(\x,{(--2.02-0*\x)/1});
\draw [domain=-1:6] plot(\x,{(--2.94-0*\x)/1});
\draw [domain=-1:6] plot(\x,{(--4-0*\x)/1});
\draw [domain=-1:6] plot(\x,{(--5-0*\x)/1});
\draw [line width=1.6pt,dash pattern=on 3pt off 3pt,domain=-1:6] plot(\x,{(--4--1*\x)/1});
\draw [line width=1.6pt,dash pattern=on 3pt off 3pt,domain=-1:6] plot(\x,{(-2--1*\x)/1});
\draw [line width=1.6pt,dash pattern=on 3pt off 3pt,domain=-1:6] plot(\x,{(-4--1*\x)/1});
\draw [line width=1.6pt,dash pattern=on 3pt off 3pt,domain=-1:6] plot(\x,{(--2--1*\x)/1});
\begin{scriptsize}
\fill [color=uququq](2,2.02) circle (1.5pt);
\fill [color=uququq] (2,2.94) circle (1.5pt);
\fill [color=uququq](2,5) circle (1.5pt);
\fill [color=uququq] (3.02,2.94) circle (1.5pt);
\fill [color=uququq] (3.02,4) circle (1.5pt);
\fill [color=uququq] (3.98,5) circle (1.5pt);
\fill [color=uququq](3.98,3.98) circle (1.5pt);
\fill [color=uququq] (3.98,2.94) circle (1.5pt);
\fill [color=uququq] (5,4) circle (1.5pt);
\fill [color=uququq] (5,2.02) circle (1.5pt);
\fill [color=uququq] (3.02,2.02) circle (1.5pt);
\end{scriptsize}
\end{tikzpicture}

\end{center}
This arrangement is free because the syzygy of degree $2$, that is, 
$$\phi=[(x-z)(x-\zeta z), (y-z)(y-\zeta z),(x-\zeta^2 y)(x-\zeta^4 y) ],$$ has no zero. 
Indeed,  for
$$\xymatrix@C=16pt{
  0\ar[r]&\T_{\cA_0}\ar[r]&  \cO^3_{\P^2}(-5) \ar[r]^-{\psi}&  \mathcal{I}_T(-1)\ar[r]&0}
$$
this syzygy gives a non-zero section $\cO_{\P^2}(-7) \to \T_{\cA_0}$, where 
$$\psi=\left[\frac{x^6-z^6}{(x-z)(x-\zeta z)}, \frac{y^6-z^6}{(y-z)(y-\zeta z)},\frac{x^6-y^6}{(x-\zeta^2 y)(x-\zeta^4 y)} \right].$$
This induces a commutative diagram
$$
\xymatrix{
 &  \cO_{\P^2}(-7) \ar[r]^{\simeq} \ar[d] & \cO_{\P^2}(-7)  \ar[d]\\
 0 \ar[r] & \T_{\cA_0} \ar[d]  \ar[r]& \cO^3_{\P^2}(-5) \ar[r] \ar[d] & \mathcal{I}_T(-1) \ar[r] \ar[d] & 0 \\
 0\ar[r] & \mathcal{I}_{\Gamma}(-7)   \ar[r] & \mathcal{F} \ar[r] & \mathcal{I}_{T}(-1) \ar[r] & 0\rlap{,} 
}
$$
where the singular locus of the rank $2$ sheaf $\mathcal{F}$ is the zero set of $\psi$. Since this zero set is empty, $\mathcal{F}$
is a vector bundle. That proves that $\mathcal{E}xt^1(\mathcal{F},\cO_{\P^2})=0$ and then $\Gamma=\emptyset$.

Another argument can be used to establish the freeness: the $12$ inner triple points are distributed as a partition $3+3+3+3$ along the vertical lines and $3+3+3+3$ along the horizontal lines but $2+3+4+3$ along the diagonal (this means that this example is very close to having the same combinatorics as the non-free case: indeed, in the following example we will see that the partition along vertical, horizontal and diagonal lines is always $3+3+3+3$). Thanks to \cite[Case 2.2]{WY}, the bundle restricted to the line containing only two inner triple points has the splitting $(7,7)$. Indeed, along the chosen diagonal line, the others cut out eight points with multiplicities $(5,2,2,1,1,1,1,1)$. The obtained splitting type implies the freeness according to \cite[Corollary 2.12]{EF}.

\subsubsection*{\it Construction of $\cA_1$}
It is obtained by removing the three lines  $x=z$,  $y=z$, $x=y$
from the full monomial arrangement  $xyz(x^5-y^5)(y^5-z^5)(x^5-z^5)=0$ as it appears in the following figure: 

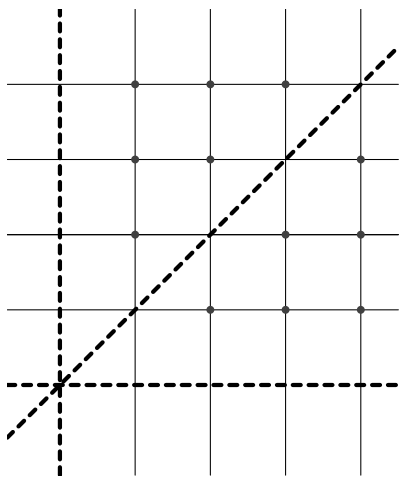
\begin{center}
 \begin{tikzpicture}[line cap=round,line join=round,>=triangle 45,x=1.0cm,y=1.0cm]
 \definecolor{uququq}{rgb}{0.25,0.25,0.25}
\definecolor{qqqqff}{rgb}{0,0,1}
\clip(0.3,-1.2) rectangle (5.5,5);
\draw [domain=0.3:5.5] plot(\x,{(--1-0*\x)/1});
\draw [domain=0.3:5.5] plot(\x,{(--2-0*\x)/1});
\draw (5,-1.2) -- (5,5);
\draw [domain=0.3:5.5] plot(\x,{(--3-0*\x)/1});
\draw (2,-1.2) -- (2,5);
\draw (3,-1.2) -- (3,5);
\draw (4,-1.2) -- (4,5);
\draw [line width=1.6pt,dash pattern=on 3pt off 3pt] (1,-1.2) -- (1,5);
\draw [domain=0.3:5.5] plot(\x,{(--4-0*\x)/1});
\draw [line width=1.6pt,dash pattern=on 3pt off 3pt,domain=0.3:5.5] plot(\x,{(-0-0*\x)/1});
\draw [domain=0.3:5.5] plot(\x,{(-10-0*\x)/-5});
\draw [line width=1.6pt,dash pattern=on 3pt off 3pt,domain=0.3:5.5] plot(\x,{(-4--4*\x)/4});
\begin{scriptsize}
\fill [color=uququq] (3,1) circle (1.5pt);
\fill [color=uququq] (4,2) circle (1.5pt);
\fill [color=uququq] (4,1) circle (1.5pt);
\fill [color=uququq] (5,1) circle (1.5pt);
\fill [color=uququq] (5,2) circle (1.5pt);
\fill [color=uququq] (2,2) circle (1.5pt);
\fill [color=uququq] (2,3) circle (1.5pt);
\fill [color=uququq] (3,3) circle (1.5pt);
\fill [color=uququq] (5,3) circle (1.5pt);
\fill [color=uququq] (2,4) circle (1.5pt);
\fill [color=uququq] (3,4) circle (1.5pt);
\fill [color=uququq] (4,4) circle (1.5pt);
\end{scriptsize}
\end{tikzpicture}
\end{center}

Beginning with $\cA_3^3(5)$, which is free with exponents $(6,11)$, and removing the first line, we obtain, by the Addition-Deletion theorem, a free bundle with exponents $(6,10)$. Removing the second line, we  again find a free arrangement, with exponents $(6,9)$.  Removing the third line, we do not find a free bundle (with splitting $(7,7)$) but find a \textit{nearly free} bundle (defined in \cite{DS} and studied by the authors in \cite{MV}) with generic splitting $(6,8)$. The jumping point is the intersection point of the three removed lines. The three partitions appearing along the horizontal, vertical and diagonal lines are $12=3+3+3+3$.

Let us make it more explicit.
We found a  syzygy of degree $1$, which is 
$$\phi=[x-z, y-z,x-y],$$ 
and which induces a non-zero section $\cO_{\P^2}(-6) \to \T$, where
$$\xymatrix@C=16pt{
  0\ar[r]&\T_{\cA_1} \ar[r]& \cO^3_{\P^2}(-5) \ar[r]^-{\psi}&  \mathcal{I}_T(-1) \ar[r]&0}
$$
 and
$$\psi=\left[\frac{x^5-z^5}{x-z}, \frac{y^5-z^5}{y-z},\frac{x^5-y^5}{x-y} \right].$$
This syzygy admits a common zero $p=(1:1:1)$ and induces a commutative diagram
$$
\xymatrix{
 &  \cO_{\P^2}(-6) \ar[r]^{\simeq} \ar[d] & \cO_{\P^2}(-6)  \ar[d]\\
 0 \ar[r] & \T_{\cA_1} \ar[d]  \ar[r]& \cO^3_{\P^2}(-5) \ar[r] \ar[d] & \mathcal{I}_T(-1) \ar[r] \ar[d] & 0 \\
 0\ar[r] & \mathcal{I}_{\Gamma}(-8)   \ar[r] & \mathcal{F} \ar[r] & \mathcal{I}_{T}(-1) \ar[r] & 0\rlap{,} 
}
$$
where the singular locus of the rank $2$ sheaf $\mathcal{F}$ is the zero set of $p$. Since $p\notin T$, we then have $p\in \Gamma$ (actually, $\Gamma=\{p\}$),  
and $\T_{\cA_1}$ cannot be free.

This example proves that the arrangement consisting of these $15$ lines is nearly free with the same weak combinatorics (same numbers $t_2,t_3, \ldots$) and almost the same combinatorics (only one partition along the diagonals differs) as the one described just before.  This shows that we cannot replace the term combinatorics with weak combinatorics in the hypothesis of Terao's conjecture.
\end{proof}

\begin{remark}
 In the famous Ziegler's example of two arrangements (nine lines with six triple points) with the same combinatorics but with different free resolutions, the situation was explained by the existence of a smooth conic containing the six triple points.  Here the situation can be geometrically explained by the existence of a cubic containing the $12$ inner triple points.  Indeed, since the bundle $\T_{\cA_1}$ described in the previous example is the kernel of the  exact sequence
$$\xymatrix@C=16pt{
   0\ar[r]&\T_{\cA_1}\ar[r]& \cO_{\P^2}^3(-5) \ar[r]&  \mathcal{I}_T(-1) \ar[r]&0,}
$$
(where $|T|=12$) it gives  $\HH^0(\mathcal{I}_T(3))=\HH^1(\T_{\cA_1}(4))$.
 Moreover,  the following non-zero global section 
$$\xymatrix@C=16pt{
   0\ar[r]&\cO_{\P^2}(-6)\ar[r]&  \T_{\cA_1} \ar[r]&  \mathcal{I}_p(-8) \ar[r]&0,}
 $$
    where $p$ is the jumping point associated to the nearly free arrangement, proves that 
    $$\hh^1(\T_{\cA_1}(4))=\hh^1(\mathcal{I}_p(-4))=\hh^0(\cO_p)=1.$$
\end{remark}

\begin{remark} 
It is possible to generalize the described examples and find a family of them in the following way: consider triangular arrangements in the family $\Tr(n,n,n)$. The multiplicity of each vertex is $n+1$.  Assume that $n=2k+1$. For arrangements in this family, the maximal possible number of inner triple points is $|T|=4k^2$ (then the arrangement is $\cA_3^3(n-1)$). For a general triangular arrangement, the set of inner triple points $T$ is empty, but if we want to consider free arrangements, $T$ must contain at least $3k^2$ points.  This minimal number corresponds to the balanced free arrangement with exponents $(3k+1,3k+1)$. We construct a nearly free arrangement with generic splitting $\cO_l(-3k)\oplus \cO_l(-2-3k)$ and $3k^2$ triple inner points in the same family $\Tr(n,n,n)$.  This generic splitting is also that of the free arrangement with $3k^2+1$ triple inner points constructed by removing $k-1$ (inner) lines from each vertex of the arrangement
$xyz(x^{3k-1}-y^{3k-1})(y^{3k-1}-z^{3k-1})(x^{3k-1}-z^{3k-1})=0$. In the last step, instead of removing a line with $k-1$ triple inner points, we remove a line with $k$ triple points.  This is always possible by choosing a line of the third direction passing through an intersection point $\{p\}$ of the previous removed lines in the two other directions. This construction induces an exact sequence
$$\xymatrix@C=16pt{
  0\ar[r]& \cO_{\P^2}(-3k)\ar[r]& \TA \ar[r]&  \mathcal{I}_p(-2-3k)\ar[r]& 0,}
$$
where, once again, $p$ is the point associated to the nearly free arrangement. 
\end{remark}

\section{Some final questions}\label{sec-Terao}

In this final section we propose some questions that would extend the study of RUAs in the more general setting of triangular arrangements. Such inquiries arise naturally from the study of large families of examples and, in our experience, should have a positive answer.

\begin{question}
Let $\cA$ be a Roots-of-Unity-Arrangement, obtained by deleting $a', b'$ and $c'$ lines from, respectively, the first, second and third family of inner lines of a full monomial arrangement. Consider a succession of RUAs
$$
\cA_0 \supset \cA_1\supset \cdots \supset \cA = \cA_{c'} \supset \cdots \supset \cA_{N-1} \supset \cA_{N},
$$
where $\cA_j$ denotes a RUA obtained from the full monomial arrangement by deleting the same $a'$ and $b'$ from the first and second family, respectively, and $j$ lines from the third one, which, depending on how many we delete, are a subset of or contain the $c'$ lines deleted to obtain $\cA$.

Suppose $\cA$ to be free. Does this imply that it possible to choose properly the order of the deleted lines to get $\cA_j$ free for any $1\leq j \leq c'-1$ or $c'+ 1\leq j \leq N-1 $? 
\end{question}

Notice that $\cA_0$ and $\cA_N$ are free arrangements; hence a  positive answer to the previous question would imply that RUAs are \textit{inductively free}, \textit{i.e.}~free at each step.

Afterwards, it is natural as well to determine if freeness is preserved when passing from a triangular arrangement to any RUA with the same intersection lattice and \textit{vice versa}. Recall that we have infinite choices  because, given a triangular arrangement, the proof of Theorem~\ref{thm-roots} implies infinitely many associated RUAs. Observe that this would give further evidence towards Terao's conjecture, as it is a special case.

\begin{question}
Let $\cA$ be a triangular arrangement and $\mathcal{B}$ a RUA with the same intersection lattice as $\cA$. Is $\cA$ free if and only if $\mathcal{B}$ is free as well?
\end{question}

\end{document}